\documentclass[12pt]{amsart}
\usepackage{amsmath,amsfonts,amsbsy,amsgen,amscd,mathrsfs,amssymb,amsthm}
\usepackage[usenames,dvipsnames]{xcolor}
\usepackage[colorlinks=true,citecolor=red,linkcolor=blue]{hyperref}
\usepackage{geometry}
\usepackage{setspace}
\allowdisplaybreaks[4]
\numberwithin{equation}{section}
\newtheorem{theorem}{Theorem}[section]

\newtheorem{lemma}[theorem]{Lemma}
\newtheorem{corollary}[theorem]{Corollary}
\newtheorem{remark}[theorem]{Remark}

\newtheorem{problem}[theorem]{Problem}

\title[curvature flow and the $(p,q)$-Christoffel-Minkowski problems]
{An expanding curvature flow and the $(p,q)$-Christoffel-Minkowski problems}

\author{Bin Chen}
\address{Bin Chen \newline \indent School of Mathematics and Statistics, Nanjing University of Science and Technology, Nanjing, China}
\curraddr{}
\email{chenb121223@163.com}
\thanks{}

\author{Jingshi Cui}
\address{Jingshi Cui \newline \indent chool of Mathematics and Statistics, Nanjing University of Science and Technology, Nanjing, China}
\curraddr{}
\email{cuijingshi626@163.com}
\thanks{The research is supported by the National Science Foundation of China (12271254; 12141104)}

\author{Peibiao Zhao}
\address{Peibiao Zhao \newline \indent School of Mathematics and Statistics, Nanjing University of Science and Technology, Nanjing, China}
\email{pbzhao@njust.edu.cn}
\thanks{Corresponding author: Peibiao Zhao}



\begin{document}
\begin{abstract}
The present paper introduces a new class of geometric measures, the $k$-th $(p,q)$-mixed curvature measures, and 
  a natural correspondence-$(p,q)$-Christoffel-Minkowski problem is proposed.
 The $(p,q)$-Christoffel-Minkowski problem posed here can be regarded as a natural generalization of the $L_p$ Christoffel-Minkowski problem and $L_p$ dual Minkowski problem.

We investigate and arrive at the existence of smooth solution to the $(p,q)$-Christoffel-Minkowski problem by a type of expanding curvature flow. Furthermore, the uniqueness result of solutions to the $(p,q)$-Christoffel-Minkowski problem shall be discussed.
\end{abstract}

\maketitle

\vskip 20pt
\section{Introduction }
It is well know that the classical Brunn-Minkowski-theory of convex bodies (i.e., compact, convex sets) in $n$-dimensional Euclidean spaces $\mathbb{R}^n$ plays an important role in the study of convex geometric analysis and develops rapidly in recent years.
The {\it classical Minkowski problem}, introduced by Minkiwski \cite{M}, is one of the cornerstones of the classical Brunn-Minkowski theory
which asks if a given Borel measure $\mu$ on the unit sphere $\mathbb{S}^{n-1}$ arises as the surface area measure $S(K,\cdot)$ of a convex body $K$ in $\mathbb{R}^n$. Here the surface area measure $S(K,\cdot)$ of $K$ is defined, for
Borel $\omega\subseteq\mathbb{S}^{n-1}$, by
$$S(K,\omega)=\int_{x\in g_{K}^{-1}(\omega)}d\mathcal{H}^{n-1}(x),$$
where $g_K: \partial^\prime K\rightarrow\mathbb{S}^{n-1}$ is the Gauss map of $K$, defined on $\partial^\prime K$, the set of points of $\partial K$ that have a unique outer unit normal, and $\mathcal{H}^{n-1}$ is $(n-1)$-dimensional Hausdorff measure.  The classical Minkowski problem argues the existence, uniqueness and regularity of a convex body.

If $\partial K$ is smooth with a positive Gauss curvature, the surface area measure of $K$ is absolutely continuous with respect to Lebesgue measure, $S$, on $\mathbb{S}^{n-1}$, and the density is the reciprocal Gauss curvature, when it is viewed as a function of the outer unit normals of $\partial K$. The density has an explicit description in terms of the support function and its Hessian matrix on $\mathbb{S}^{n-1}$,
$$\frac{dS(K,\cdot)}{dS}=\det(\nabla_{ij}h_K+h_K\delta_{ij}),$$
where $h_K$ is the support function of $K$, $\nabla_{ij}h$ is the Hessian matrix of $h$ w.r.t. an orthonormal frame on $\mathbb{S}^{n-1}$, and $\delta_{ij}$ is the Kronecker symbol.
Thus, if the given Borel measure $\mu$ has a positive continuous density, the classical Minkowski problem can be seen as the problem of prescribing the Gauss curvature in differential geometry.

In modern convex geometry, the $L_p$ Minkowski problem \cite{L}, Orlicz Minkowski problem \cite{H0} and their dual Minkowski problem \cite{G1,G2,HLY,L0} generalize and dualize the classical Minkowski problem, and then studied by \cite{Bo,Cw,HL0,HZ,J,J1,LYZ0,Z,Z1,Z2,Z3,Z4} and  the references therein.

In \cite{L}, Lutwak defined the {\it $k$-th $L_p$-surface area measure} using $L_p$ variational formula of quermassintegral:
\begin{align}\label{1.0.1}
\nonumber S_{p,k}(K,\omega)&=\int_\omega h_K^{1-p}dS_{k}(K,\cdot)\\
&=\int_\omega\sigma_{n-k}(\kappa_K)dS_{p}(K,\cdot),\ \forall \ Borel \ set\ \omega\subseteq\mathbb{S}^{n-1},
\end{align}
for $K\in\mathcal{K}_o^n$ (the set of convex bodies containing the origin in its interior) and $k=1,\cdots,n-1$, where $S_k(K,\cdot)$ is the {\it $k$-th surface area measure}, $S_{p}(K,\cdot)$ is the $L_p$-surface area measure, $\sigma_{n-k}$ is the $(n-k)$-th elementary symmetric function and $\kappa_K$ are the principal curvatures of $K$.

The general problem concerns with the existence of convex bodies with prescribed $k$-th $L_p$-surface area measure is often called the {\it $L_p$-Christoffel-Minkowski problem}.
Inspired by the $L_p$ (Orlicz)-Minkowski problem, it is natural to consider the $L_p$ (Orlicz)-Christoffel-Minkowski problem (see e.g., \cite{C0,GM,GX,HMS,JL}).

Motivated by the foregoing celebrated works,  we introduce  a new class of geometric measures $\mathcal{M}_{(p,q),k}(K,\cdot)$ as follows:
\begin{align}\label{1.0.2}
\nonumber\mathcal{M}_{(p,q),k}(K,\omega)&=\int_\omega h_K^{-p}d\tilde{C}_{q,k}(K,\cdot)\\
&=\int_\omega\sigma_{n-k}^{-1}(\lambda_K)d\tilde{C}_{p,q}(K,\cdot),\ \forall \ Borel \ set\ \omega\subseteq\mathbb{S}^{n-1},
\end{align}
for $K\in\mathcal{K}_o^n$ and $k=1,\cdots,n-1$, where $\lambda_K$ are the principal radii of curvature of $K$, $\tilde{C}_{q,k}(K,\cdot)$ and $\tilde{C}_{p,q}(K,\cdot)$ are the $k$-th dual curvature measure and $L_p$ dual curvature measure introduced in \cite{LSW,L0}. We refer to the measure $\mathcal{M}_{(p,q),k}(K,\cdot)$ as the {\it $k$-th $(p,q)$-mixed curvature measure}.

Naturally, we state the Christoffel-Minkowski type problem as follows.
\begin{problem}\label{p1.1}
(The $(p,q)$-Christoffel-Minkowski problem) For fixed $p,q\in\mathbb{R}$ and $k=1,\cdots,n-1$. What are the necessary and sufficient conditions for a Borel measure $\mu$ on $\mathbb{S}^{n-1}$ to be the $k$-th $(p,q)$-mixed curvature measure $\mathcal{M}_{(p,q),k}(K,\cdot)$ of $K\in\mathcal{K}_o^n$? If such a convex body $K$ exists to what extent is $K$ unique?
\end{problem}

It is well known that $\tilde{C}_{p,q}(K,\cdot)$ in the smooth category is absolutely continuous with respect to  Lebesgue measure $S$ on $\mathbb{S}^{n-1}$, and $d\tilde{C}_{p,q}(K,\cdot)/dS=h^{1-p}\rho^{q-n}\det(\nabla_{ij}h+h\delta_{ij})$, where $h, \rho$ are the support and radial functions of $K$. If the given measure $\mu$ has a density function $f$, then Problem \ref{p1.1} is equivalent to solving the following generalized Hessian equation
\begin{align}\label{1.0}
\frac{h^{1-p}}{(h^2+|\nabla h|^2)^{\frac{n-q}{2}}}\sigma_k(\nabla_{ij}h+h\delta_{ij})=f.
\end{align}

When $1\leq k<n-1$ and $q=n$, Equation (\ref{1.0}) is known as the $L_p$-Christoffel-Minkowski problem and is the classical Christoffel-Minkowski problem for $p=1$ (\cite{GM}). For the case of $p>1$ and $q=n$, Equation (\ref{1.0}) has been studied by Hu-Ma-Shen \cite{HMS} for $p\geq k+1$, and by Guan-Xia \cite{GX} for $1<p<k+1$ via also the constant rank theorem;
For the case of $p<1$, Chen \cite{C0} gave a uniqueness of solution via flows by powers of Gauss curvatures.
As far as we know, the existence for $p<1$ are unknown until now. An extreme case of Equation (\ref{1.0}) is $k=n-1$, which corresponds to the $L_p$ dual Minkowski problem introduced in \cite{HLY}, and then followed by \cite{CTW,CH,CL,L0} and the references therein. Moreover, the case of $k=n-1$ and $q=n$ is the $L_p$ Minkowski problem posed in \cite{L}.

In this paper, we confirm the existence of smooth solutions to the $(p,q)$-Christoffel-Minkowski problem by the curvature flow method.

Curvature flows of convex hypersurfaces in $\mathbb{R}^n$ by a class of speed functions of the principal curvatures have been extensively studied in the past decades. The well-known examples include the Gauss curvature flow (see e.g., \cite{C,F1}), and the mean curvature flow (see e.g., \cite{Hui}). The reason why geometers are interested in the study of the theory is that it has important applications in physics and mathematics. In recent years, the curvature flow technique has been proved to be effective and powerful in solving the Minkowski type problems (see, e.g., \cite{B1,CH,CW,FLX,I,LL,LS,LSW,SY}).
The essential idea behind the flow technique is the fact that the Minkowski type problems can be reformulated as a Monge-Amp\`{e}re type equation.

The main purpose of this paper is to construct a suitable expanding curvature flow of convex hypersurfaces in  $\mathbb{R}^n$, and prove its long-time existence and convergence smoothly to a smooth solution of Equation (\ref{1.0}).

Let $M_0$ be a smooth, closed, and strictly convex hypersurface containing the origin $o$ in its interior, that is, there is a sufficient small positive constant $\delta_o$ such that the $\delta_0$-neighbourhood of $o$ being with  $U(o,\delta_o)\subset M_0$. We now consider and write down an expanding flow as follows.
For a family of closed hypersurfaces $\{M_t\}$ given by $M_t=X(\mathbb{S}^{n-1},t)$,
where $X:\mathbb{S}^{n-1}\times [0,T)\rightarrow\mathbb{R}^n$ is the smooth map that satisfies
\begin{align}\label{1.1}
\begin{cases}
\frac{\partial X}{\partial t}(x,t)=\frac{1}{f(\nu)}\frac{\langle X,\nu\rangle^{1-p}}{|X|^{n-q}}\sigma_k(x,t)\nu-X(x,t),\\
X(x,0)=X_0(x),
\end{cases}
\end{align}
for integer $1\leq k\leq n-1$, where $f$ is a given positive smooth function on $\mathbb{S}^{n-1}$, $\nu$ is the unit outer normal vector of $M_t$, $\langle\cdot,\cdot\rangle$ is the standard inner product in $\mathbb{R}^n$, $\sigma_k$ is the $k$-th elementary symmetric function for principal curvature radii, and $T$ is the maximal time for which the solution of (\ref{1.1}) exists.

The following theorem obtains the long-time existence and convergence of the flow (\ref{1.1}). Moreover, we will derive that a positive homothetic self-similar solution of the flow (\ref{1.1}) satisfies Equation (\ref{1.0}).

\begin{theorem}\label{th1.1}
Let $M_0$ be a smooth, closed, and strictly convex hypersurface containing the origin o in its interior. Suppose $f$ is a positive smooth function on $\mathbb{S}^{n-1}$, and $p,q\in\mathbb{R}$ with $q<p$. Then the flow (\ref{1.1}) has a smooth, closed, and strictly convex solution $M_t$, which exists for any time $t\in[0,\infty)$.
Moreover, when $t\rightarrow\infty$, a subsequence of $M_t$ converges in $C^\infty$ to a smooth, closed, and strictly convex hypersurface $M_\infty$, and the support function of $M_\infty$ satisfies Equation (\ref{1.0}).
\end{theorem}

As an application, we have

\begin{corollary}\label{co1.2}
Under the assumptions of Theorem \ref{th1.1}, there exists a smooth solution to Equation (\ref{1.0}). In other words, the $(p,q)$-Christoffel-Minkowski problem has a smooth solution.
\end{corollary}

\begin{remark}\label{rem1.3}
When $q=n$ in (\ref{1.0}), the existence of smooth solution to the $L_p$ Christoffel-Minkowski problem for $p\neq0$ is obtained follows Corollary \ref{co1.2}.
When $k=n-1$ in (\ref{1.0}), we obtain the existence of smooth solutions to the $L_p$ dual Minkowski problem for $p,q\in\mathbb{R}$ with $q<p$.
\end{remark}

Finally, we consider the uniqueness result of Equation (\ref{1.0}).

\begin{theorem}\label{th1.4}
If $p,q\in\mathbb{R}$ with $q<p$, then the solution to Equation (\ref{1.0}) is unique.
\end{theorem}

This paper is organized as follows.
The corresponding background materials and some results are introduced in Section \ref{S2}.
In Section \ref{S3}, we establish the priori estimates for the solution to the flow (\ref{1.1}).
In Section \ref{S4}, we complete the proof of Theorem \ref{th1.1}.
In Section \ref{S5}, we provide a uniqueness result of Equation (\ref{1.0}).

\section{ Preliminaries}\label{S2}

\subsection{Convex hypersurfaces}

In this subsection, we list some facts about convex hypersurfaces that readers can refer to \cite{U} and two good books of Gardner and Schneider \cite{G,S}.
Let $\mathbb{R}^n$ be the $n$-dimensional Euclidean
space, $\mathbb{S}^{n-1}$ be the unit sphere in $\mathbb{R}^n$.
For $K\in\mathcal{K}_o^n$, the support function $h:
\mathbb{S}^{n-1}\rightarrow\mathbb{R}$ of $K$ is defined by
\begin{align*}
h(x)=\max\{\langle x,Y\rangle,\ Y\in K\},\  x\in\mathbb{S}^{n-1}.
\end{align*}
The radial function of $K$ be $\rho$ and defined as
\begin{align*}
\rho(u)=\max\{\lambda>0,\lambda u\in K\},\ \ u\in \mathbb{S}^{n-1}.
\end{align*}

Let $M$ be a smooth, closed, strictly convex hypersurface containing the origin $o$ in its interior in $\mathbb{R}^n$. Denote the Gauss map of $M$ by $g_M$.
Assume that $M$ is given by the smooth map $X:
\mathbb{S}^{n-1}\rightarrow M\subset\mathbb{R}^n$ with $X(x)=g_M^{-1}(x)$.
The maximun of $h(x)$ is attained at the end of $Y$, hence
\begin{align}\label{2.0}
h(x)=\langle x,X(x)\rangle, \ \ x\in\mathbb{S}^{n-1},
\end{align}

Let $e_{ij}$ be the standard metric of the unit sphere $\mathbb{S}^{n-1}$,
and $\nabla$ be the gradient on $\mathbb{S}^{n-1}$.
Differentiating (\ref{2.0}), we have
$$\nabla_ih=\langle\nabla_ix,X(x)\rangle+\langle x,\nabla_iX(x)\rangle,$$
since $\nabla_iX(x)$ is tangent to $M$ at $X$, then
$$\nabla_ih=\langle\nabla_ix,X(x)\rangle.$$
It follows that
\begin{align}\label{2.01}
X(x)=\nabla h(x)+h(x)x.
\end{align}
From (\ref{2.01}), $u$ and $x$ are related by
\begin{align}\label{2.02}
\rho(u)u=h(x)x+\nabla h(x).
\end{align}

By differentiating (\ref{2.0}) twice, the second fundamental form of $M$ is given  by
\begin{align}\label{2.1}
A_{ij}=\nabla_{ij}h+he_{ij},
\end{align}
where $\nabla_{ij}=\nabla_i\nabla_j$ denotes the second order covariant derivative with respect to $e_{ij}$. The induced metric matix $g_{ij}$ of $M$ can be derived by Weingarten's formula,
\begin{align}\label{2.2}
e_{ij}=\langle\nabla_ix,\nabla_jx\rangle=A_{ik}A_{lj}g^{kl}.
\end{align}
It follows from (\ref{2.1}) and (\ref{2.2}) that the principal radii of curvature of $M$, under a smooth local orthonormal frame on $\mathbb{S}^{n-1}$, are the eigenvalues of the matrix
\begin{align}\label{2.21}
b_{ij}=\nabla_{ij}h+h\delta_{ij}.
\end{align}
Since $M$ be a smooth, closed, strictly convex hypersurface containing the origin in its interior, then $h(M,\cdot)>0$, and the eigenvalues of $b_{ij}$ are positive. Thus $b_{ij}$ is positive definite. Further, we will use $b^{ij}$ to denote the inverse matrix of $b_{ij}$.

\subsection{Curvature flow and its associated functional}

By the definition of support function, it is easy for us to see
$h(x,t)=\langle x,X(x,t)\rangle$. From the evolution equation of
$X(x,t)$ in (\ref{1.1}), we derive the parameterized  evolution
equation by the corresponding support function $h(x,t)$ as
\begin{align}\label{2.3}
\frac{\partial h}{\partial t}(x,t)=\frac{1}{f(x)}\frac{h^{2-p}}{\rho^{n-q}}
\sigma_k(x,t)-h(x,t).
\end{align}
Let $x$ be expressed as $x=x(u,t)$, by (\ref{2.02}), we get
\begin{align*}
\log\rho(u,t)=\log h(x,t)-\log\langle x,u\rangle.
\end{align*}
Differentiating the above identity, it is easy to see that there
holds
\begin{align}\label{2.3.1}
\frac{1}{\rho(u,t)}\frac{\partial\rho(u,t)}{\partial t}=\frac{1}{h(x,t)}\frac{\partial h(x,t)}{\partial t}.
\end{align}
Therefore, by (\ref{2.3.1}), the flow (\ref{1.1}) can be aslo described by the following scalar equation for $\rho(u,t)$.
\begin{align}\label{2.3.2}
\frac{\partial \rho}{\partial t}(u,t)=\frac{1}{f(u)}\frac{h^{1-p}}{\rho^{n-q-1}}\sigma_k(u,t)-\rho(u,t).
\end{align}

\section{ A priori estimation }\label{S3}
In this section, we give the $C^0$, $C^1$ and $C^2$-estimates for
the solution to the equation (\ref{2.3}).

\subsection{$C^0$, $C^1$-Estimates}

The following Lemma shows the $C^0$-estimate.

\begin{lemma}\label{la3.1}
Let $h(\cdot,t)$ be a smooth solution of (\ref{2.3}) on $\mathbb{S}^{n-1}\times[0,T)$.
Suppose $f$ is a positive smooth function on $\mathbb{S}^{n-1}$, and
 $p,q\in\mathbb{R}$ with $q<p$. Then there holds
\begin{align}\label{3.1}
C_1\leq h(\cdot,t)\leq C_2,
\end{align}
and
\begin{align}\label{3.2}
C_1\leq \rho(\cdot,t)\leq C_2,
\end{align}
where $C_1$ and $C_2$ are positive constants independent of $t$.
\end{lemma}

\begin{proof}
From the definitions of support function and radial function, we know
$$\min_{\mathbb{S}^{n-1}}h(x,t)
=\min_{\mathbb{S}^{n-1}}\rho(u,t),\ \
\max_{\mathbb{S}^{n-1}}h(x,t)
=\max_{\mathbb{S}^{n-1}}\rho(u,t).$$
This implies that (\ref{3.1}) and (\ref{3.2}) are equivalent. Thus, for upper bound (or lower bound), we only need to establish (\ref{3.1}) or (\ref{3.2}).

Suppose the spatial minimum of $h(x,t)$ is attained at a point $(x_0,t)$ for each $t\in[0,T)$. At this point, we have
$$\nabla_ih=0,\ \ \nabla_{ij}h\geq0,\ \ and \ \ \rho=h.$$
It follows that $\sigma_k\geq h_{\min}^k$.
Under the assumptions of $f$, there is
$$\partial_th_{\min}(t)\geq\frac{1}{f}
h^\alpha_{\min}-h_{\min}
\geq h_{\min}((c^{\frac{1}
{1-\alpha}}h_{\min})^{\alpha-1}-1),$$
where $\alpha=k+q+2-n-p$.
Hence $h_{\min}\geq\min\{1,c^{\frac{1}
{1-\alpha}}h_{\min}(0)\}$.

Suppose the spatial maximum of $h(x,t)$ is attained at a point $(x^0,t)$ for each $t\in[0,T)$. At this point, we have
$$\nabla_ih=0,\ \ \nabla_{ij}h\leq0,\ \ and \ \ \rho=h.$$
It follows that $\sigma_k\leq h_{\max}^k$.
Similarly,
$$\partial_th_{\max}(t)\leq\frac{1}{f}
h^\alpha_{\max}-h_{\max}
\leq h_{\max}((c^{\frac{1}{\alpha-1}}
h_{\max})^{\alpha-1}-1).$$
Hence $h_{\max}\leq\min\{1,c^{\frac{1}{\alpha-1}}h_{\max}(0)\}$.
\end{proof}

Since the convexity of $M_t$, combining with Lemma \ref{la3.1}, we can obtain the $C^1$-estimate as follows.
\begin{lemma}\label{la3.2}
Under the assumptions of Lemma \ref{la3.1}, we have
$$|\nabla h(\cdot,t)|\leq C,$$
where $C$ is a positive constant depending only on the constant in Lemma \ref{la3.1}.
\end{lemma}

\begin{proof}
From the equality (\ref{2.02}), we can get
$$\rho^2=h^2+|\nabla h|^2.$$
By virtue of the Lemma \ref{la3.1}, there is
$$|\nabla h|\leq\rho.$$
So we directly obtain the estimate of this lemma.
\end{proof}

\subsection{$C^2$-Estimate}

In this subsection, we will establish the upper and lower bounds of principal curvatures. This shows that the equation (\ref{2.3}) is uniformly parabolic.
We first start from the lower bound of $\sigma_k(\cdot,t)$.

\begin{lemma}\label{la3.4}
Under the assumptions of Lemma \ref{la3.1}.  Then
$\sigma_k(\cdot,t)\geq C_3$ for a positive constant $C_3$ independent of $t$.
\end{lemma}

\begin{proof}
In order to obtain the lower bound of $\sigma_k$, we will apply the maximum principle
to the following auxiliary function, a similar function is considered in \cite{KS},
$$\Theta(x,t)=\log Q-A\frac{\rho^2}{2},$$
where $A>0$ is a constant to be determined later and $Q=N\sigma_k=\frac{\sigma_k}{f(x)}h^{2-p}\rho^{q-n}$.
The evolution equation of $\Theta$ reads as
\begin{align}\label{3.3}
\partial_t\Theta=\frac{1}{Q}\partial_tQ-A\partial_t\bigg(\frac{\rho^2}{2}\bigg).
\end{align}

We now calculate the evolution equation of $Q$
\begin{align*}
\partial_tQ=N\partial_t\sigma_k+\sigma_k\partial_tN.
\end{align*}
Here
\begin{align*}
\sigma_k\partial_tN&=\bigg[\frac{(2-p)h^{1-p}
\partial_th}{f\rho^{n-q}}-
\frac{(n-q)h^{2-p}f\rho^{n-q-1}
\partial_t\rho}{(f\rho^{n-q})^2}\bigg]\sigma_k\\
&=\bigg[\frac{(2-p)h^{1-p}}{f\rho^{n-q}}
-\frac{(n-q)h^{1-p}}{f\rho^{n-q}}\bigg]\sigma_k\partial_th\\
&=\beta\frac{Q}{h}(Q-h),
\end{align*}
where $\beta=q+2-n-p$.
From (\ref{2.21}), and denote partial derivatives $\partial\sigma_k/\partial b_{ij}$ by $\sigma_k^{ij}$, then we get
\begin{align*}
\partial_t\sigma_k=\frac{\partial\sigma_k}{\partial b_{ij}}\frac{\partial b_{ij}}{\partial t}=\sigma^{ij}_k\partial_t(\nabla_{ij}h+h\delta_{ij}),
\end{align*}
Thus we arrive at
\begin{align*}
\partial_t\sigma_k
&=\sigma^{ij}_k\nabla_{ij}(\partial_th)
+\sigma^{ij}_k\delta_{ij}(\partial_th)\\
&=\sigma^{ij}_k\nabla_{ij}(Q-h)
+\sigma^{ij}_k\delta_{ij}(Q-h)\\
&=\sigma^{ij}_k\nabla_{ij}Q
+\sigma^{ij}_k\delta_{ij}Q-\sigma^{ij}_k(\nabla_{ij}h
+h\delta_{ij})\\
&=\sigma^{ij}_k\nabla_{ij}Q
+\sigma^{ij}_k\delta_{ij}Q-k\sigma_k,
\end{align*}
the last equality holds because $\sigma_k$ is homogeneous of degree $k$ and $\sigma_k^{ij}b_{ij}=k\sigma_k$.

Hence
\begin{align}\label{3.4}
\partial_tQ=N\sigma_k^{ij}\nabla_{ij}Q+N\sigma_k^{ij}
\delta_{ij}Q-(k+\beta)Q+\beta\frac{Q^2}{h}.
\end{align}

Recalling that $\rho^2=h^2+|\nabla h|^2$, we have
\begin{align}\label{3.5}
\nonumber\partial_t\bigg(\frac{\rho^2}{2}\bigg)
&=\partial_t\bigg(\frac{h^2}{2}\bigg)
+\partial_t\bigg(\frac{|\nabla h|^2}{2}\bigg)\\
&=h\partial_th+\nabla_ih\nabla_i(\partial_th)\\
&\nonumber=hQ+\nabla_ih\nabla_iQ-\rho^2.
\end{align}
Combines (\ref{3.4}) and (\ref{3.5}),  we have
\begin{align*}
\partial_t\Theta&=\frac{N}{Q}\sigma_k^{ij}\nabla_{ij}Q
+\frac{N}{Q}\sigma_k^{ij}\delta_{ij}Q-(k+\beta)
+\beta\frac{Q}{h}\\
&\ \ -AhQ-A\nabla_ih\nabla_iQ+A\rho^2.
\end{align*}
Suppose the spatial minimum of $\Theta$ is attained at a point $(\tilde{x},t)$, then $\nabla_{ij}Q\geq0$. At point $(\tilde{x},t)$,
dropping some positive terms and rearranging terms yield
\begin{align*}
\partial_t\Theta&\geq
A\rho^2-(k+\beta)+\beta\frac{Q}{h}
-AhQ-A\nabla_ih\nabla_iQ\\
&=\frac{A}{2}\rho^2-(k+\beta)
+\beta\frac{1}{h}e^{\Theta+A\frac{\rho^2}{2}}\\
&\ \ +A\bigg(\frac{\rho^2}
{2e^{\Theta+A\frac{\rho^2}{2}}}
-hQ-\nabla_ih\nabla_iQ\bigg).
\end{align*}
Now choose $A>\frac{2}{\rho^2}(k+\beta)$. Thus if $\Theta$ becomes very negative then the right-hand side
becomes positive, and the lower bound of $\Theta$ follows.
\end{proof}

\begin{lemma}\label{la3.5}
Under the assumptions of Lemma \ref{la3.1}, then
$\sigma_k(\cdot,t)\leq C_4$,
where $C_4$ is a positive constant independent of $t$.
\end{lemma}

\begin{proof}
We will apply the maximum principle to the auxiliary function
$$\mathcal{P}(x,t)=\frac{1}{1-\lambda\frac{\rho^2}{2}}
\frac{N\sigma_k}{h},$$
where $\lambda$ is a positive constant such that $\lambda\leq\rho^2\leq\frac{1}{\lambda}$ for all $t>0$ (know from Lemma \ref{la3.1}).
Suppose the spatial maximum of $\mathcal{P}$ is obtained at $(\hat{x},t)$. Then, at $(\hat{x},t)$,
\begin{align}\label{3.6}
\nabla_i\mathcal{P}=0,\ \ i.e.,\
\nabla_i\bigg(\frac{N\sigma_k}{h}\bigg)+\frac{N\sigma_k}{h}
\frac{\lambda}{1-\lambda\frac{\rho^2}{2}}
\nabla_i\bigg(\frac{\rho^2}{2}\bigg)=0,
\end{align}
and
\begin{align}\label{3.7}
\nabla_{ij}\mathcal{P}\leq0.
\end{align}

Now we estimate $\mathcal{P}$, using (\ref{3.7}), we have
\begin{align}\label{3.8}
\nonumber\partial_t\mathcal{P}&\leq\partial_t\mathcal{P}
-N\sigma_k^{ij}\nabla_{ij}\mathcal{P}\\
\nonumber&=\partial_t\bigg(\frac{1}{1-\lambda\frac{\rho^2}{2}}
\frac{N\sigma_k}{h}\bigg)
-N\sigma_k^{ij}\nabla_{ij}\bigg(\frac{1}{1-\lambda
\frac{\rho^2}{2}}\frac{N\sigma_k}{h}\bigg)\\
\nonumber&=\frac{1}{1-\lambda\frac{\rho^2}{2}}
\bigg[\partial_t\bigg(\frac{N\sigma_k}{h}\bigg)
-N\sigma_k^{ij}\nabla_{ij}\bigg(\frac{N\sigma_k}{h}\bigg)\bigg]\\
&\ \ +\frac{\lambda}{\bigg(1-\lambda\frac{\rho^2}{2}\bigg)^2}
\frac{N\sigma_k}{h}\bigg[\partial_t\bigg(\frac{\rho^2}{2}\bigg)
-N\sigma_k^{ij}\nabla_{ij}\bigg(\frac{\rho^2}{2}\bigg)\bigg]\\
\nonumber&\ \ -2N\sigma_k^{ij}
\frac{\lambda}{\bigg(1-\lambda\frac{\rho^2}{2}\bigg)^2}
\nabla_i\bigg(\frac{N\sigma_k}{h}\bigg)
\nabla_j\bigg(\frac{\rho^2}{2}\bigg)\\
\nonumber&\ \ -2N\sigma_k^{ij}
\frac{\lambda^2}{\bigg(1-\lambda\frac{\rho^2}{2}\bigg)^3}
\frac{N\sigma_k}{h}\nabla_i\bigg(\frac{\rho^2}{2}\bigg)
\nabla_j\bigg(\frac{\rho^2}{2}\bigg).
\end{align}
Substituting (\ref{3.6}) into (\ref{3.8}), we have
\begin{align}\label{3.9}
\nonumber\partial_t\mathcal{P}&\leq
\frac{1}{1-\lambda\frac{\rho^2}{2}}
\bigg[\partial_t\bigg(\frac{N\sigma_k}{h}\bigg)
-N\sigma_k^{ij}\nabla_{ij}\bigg(\frac{N\sigma_k}{h}\bigg)\bigg]\\
&\ \ +\frac{\lambda}{\bigg(1-\lambda\frac{\rho^2}{2}\bigg)^2}
\frac{N\sigma_k}{h}\bigg[\partial_t\bigg(\frac{\rho^2}{2}\bigg)
-N\sigma_k^{ij}\nabla_{ij}\bigg(\frac{\rho^2}{2}\bigg)\bigg].
\end{align}

We need to calculate $\partial_t(\frac{N\sigma_k}{h})
-N\sigma_k^{ij}\nabla_{ij}(\frac{N\sigma_k}{h})$ and $\partial_t(\frac{\rho^2}{2})
-N\sigma_k^{ij}\nabla_{ij}(\frac{\rho^2}{2})$.
\begin{align*}
\partial_t(N\sigma_k)&=\sigma_k\partial_tN+N\partial_t\sigma_k\\
&=\frac{\sigma_k}{f}\bigg(\frac{(2-p)h^{1-p}}{\rho^{n-q}}
\partial_th-\frac{(n-q)h^{1-p}}{\rho^{n-q}}
\partial_th\bigg)\\
&\ \ +N(\sigma_k^{ij}\nabla_{ij}(N\sigma_k)+\sigma_k^{ij}
\delta_{ij}(N\sigma_k)-k\sigma_k)\\
&=\beta\frac{(N\sigma_k)^2}{h}-\beta N\sigma_k
+N(\sigma_k^{ij}\nabla_{ij}(N\sigma_k)+\sigma_k^{ij}
\delta_{ij}(N\sigma_k)-k\sigma_k),
\end{align*}
and
\begin{align*}
\nabla_{ij}\bigg(\frac{N\sigma_k}{h}\bigg)
&=\frac{1}{h}\nabla_{ij}(N\sigma_k)
-\frac{1}{h^2}(N\sigma_k)\nabla_{ij}h\\
&\ \ -\frac{2}{h^2}\nabla_i(N\sigma_k)\nabla_jh
+\frac{2}{h^3}(N\sigma_k)\nabla_ih\nabla_jh.
\end{align*}
Hence
\begin{align*}
&\partial_t\bigg(\frac{N\sigma_k}{h}\bigg)
-N\sigma_k^{ij}\nabla_{ij}\bigg(\frac{N\sigma_k}{h}\bigg)\\
&=\frac{1}{h}\partial_t(N\sigma_k)-\frac{1}{h^2}N\sigma_k
\partial_th-N\sigma_k^{ij}\nabla_{ij}
\bigg(\frac{N\sigma_k}{h}\bigg)\\
&=\frac{1}{h}\bigg(\beta\frac{(N\sigma_k)^2}{h}-\beta N\sigma_k+N(\sigma_k^{ij}\nabla_{ij}(N\sigma_k)
+\sigma_k^{ij}\delta_{ij}(N\sigma_k)-k\sigma_k)\bigg)\\
&\ \ -\frac{1}{h^2}N\sigma_k(N\sigma_k-h)
-N\sigma_k^{ij}\bigg(\frac{1}{h}\nabla_{ij}(N\sigma_k)
-\frac{1}{h^2}(N\sigma_k)\nabla_{ij}h\\
&\ \ -\frac{2}{h^2}\nabla_i(N\sigma_k)\nabla_jh
+\frac{2}{h^3}(N\sigma_k)\nabla_ih\nabla_jh\bigg)\\
&=(\beta+k-1)\bigg(\frac{N\sigma_k}{h}\bigg)^2
-(\beta+k-1)\frac{N\sigma_k}{h}
-\frac{2N}{h^3}\sigma_k^{ij}(N\sigma_k)\nabla_ih\nabla_jh.
\end{align*}

Recalling that $\rho^2=h^2+|\nabla h|^2$, we have
\begin{align*}
&\partial_t\bigg(\frac{\rho^2}{2}\bigg)-N\sigma_k^{ij}
\nabla_{ij}\bigg(\frac{\rho^2}{2}\bigg)\\
&=h\partial_th+\nabla_ih\nabla_i(\partial_th)
-N\sigma_k^{ij}(h\nabla_{ij}h+\nabla_ih\nabla_jh+
\nabla_mh\nabla_j\nabla_{mi}h+\nabla_{mi}h\nabla_{mj}h)\\
&=(k+1)hN\sigma_k-\rho^2+\sigma_k\nabla_ih\nabla_iN
-N\sigma_k^{ij}b_{mi}b_{mj}.
\end{align*}

Thus, by (\ref{3.6})
\begin{align*}
\partial_t\mathcal{P}
&\leq\frac{1}{1-\lambda\frac{\rho^2}{2}}\bigg[
\frac{\beta+k-1}{h^2}(N\sigma_k)^2
-\frac{(\beta+k-1)}{h}N\sigma_k
-\frac{2N}{h^3}\sigma_k^{ij}(N\sigma_k)\nabla_ih\nabla_jh
\bigg]\\
&\ \ +\frac{\lambda}{\bigg(1-\lambda\frac{\rho^2}{2}\bigg)^2}
\frac{N\sigma_k}{h}\bigg[(k+1)hN\sigma_k
-\rho^2+\sigma_k\nabla_ih\nabla_iN
-N\sigma_k^{ij}b_{mi}b_{mj}\bigg].
\end{align*}
Due to the inverse concavity of $(\sigma_k)^{\frac{1}{k}}$, we have the fact $((\sigma_k)^{\frac{1}{k}})^{ij}b_{im}b_{jm}
\geq(\sigma_k)^{\frac{2}{k}}$ (see \cite{AMZ}), which means $\sigma_k^{ij}b_{im}b_{jm}\geq k(\sigma_k)^{1+\frac{1}{k}}$. Then we have
\begin{align*}
\partial_t\mathcal{P}
&\leq\frac{1}{1-\lambda\frac{\rho^2}{2}}
\bigg(\frac{\beta+k-1}{h^2}Q^2-\frac{\beta+k-1}{h}Q\bigg)\\
&\ \ +\frac{\lambda Q}{(1-\lambda\frac{\rho^2}{2})^2}
\bigg((k+1)Q+\frac{Q}{hN}\nabla_ih\nabla_iN
-\frac{k}{h}(\frac{1}{N})^{\frac{1}{k}}Q^{1+\frac{1}{k}}\bigg).
\end{align*}
From the Lemma \ref{la3.1} and Lemma \ref{la3.2}, there exists some positive constants  $c_1$, $c_2$ and $c_3$ independent of $t$ such that
$$\partial_t\mathcal{P}\leq c_1\mathcal{P}+c_2\mathcal{P}^2
-c_3\mathcal{P}^{2+\frac{1}{k}}<0$$
provided $\mathcal{P}$ is sufficiently large. Thus $\mathcal{P}(x,t)$ is uniformly bounded from above. The upper bound of $\sigma_k$ follows from the uniformly bounds on $f, h$ and $\rho$.
\end{proof}

From  Lemma \ref{la3.1},
as discussed in Sect.\ref{S2} (or see \cite{CMY,U}), we know that the eigenvalues of $\{b_{ij}\}$ and $\{b^{ij}\}$ are respectively the principal radii and principal curvatures of $M_t$, where $\{b^{ij}\}$ is the inverse matrix of $\{b^{ij}\}$.
Therefore, to derive a positive upper bound of principal curvatures of $X(\cdot,t)$,
it is equivalent to estimate the upper bound of the eigenvalues of $\{b^{ij}\}$.

\begin{lemma}\label{la3.6}
For $1\leq k\leq n-1$, $k$ is an integer.
Under the assumptions of Lemma \ref{la3.1}, then the principal curvatures $\kappa_i$ satisfies
\begin{align*}
\kappa_i(\cdot,t)\leq C_5,\ \ i=1,\cdots,n-1,
\end{align*}
where $C_5$ is a positive constant independent of $t$.
\end{lemma}

\begin{proof}
For any fixed $t\in[0,T)$, we assume that the spatial maximum of maximum eigenvalue of matrix $\{\frac{b^{ij}}{h}\}$ attained at a point $x_t$ in the direction of the unit vector $e_1\in T_{x_t}\mathbb{S}^{n-1}$. By rotation, we also choose the orthonormal vector field such that $b^{ij}$ is diagonal
and the maximum eigenvalue of $\{\frac{b^{ij}}{h}\}$ is $\frac{b^{11}}{h}$.

Now, we first give the evolution equations of $b_{ij}$ and $b^{ij}$. For convenience, set $Q=\frac{1}{f}\frac{h^{2-p}}{\rho^{n-q}}$.
From the fact $b_{ij}=\nabla_{ij}h+h\delta_{ij}$, we have
\begin{align*}
\nabla_{ij}(\partial_th)=\sigma_k\nabla_{ij}Q
+\nabla_i\sigma_k\nabla_jQ
+\nabla_j\sigma_k\nabla_iQ+Q\nabla_{ij}\sigma_k-\nabla_{ij}h,
\end{align*}
where $\nabla_{ij}\sigma_k=\sigma_k^{ls,mn}\nabla_jb_{ls}\nabla_{i}b_{mn}+\sigma_k^{ls}\nabla_{ij}b_{ls}$. By the Gauss equation
$$\nabla_{ij}b_{ls}=\nabla_{ls}b_{ij}+\delta_{ij}\nabla_{ls}h-\delta_{ls}\nabla_{ij}h+\delta_{is}\nabla_{lj}h
-\delta_{lj}\nabla_{is}h,$$
then
\begin{align*}
\nabla_{ij}(\partial_th)&=Q\sigma_k^{ls,mn}
\nabla_jb_{ls}\nabla_{i}b_{mn}
+Q\sigma_k^{ls}\nabla_{ls}b_{ij}
+kQ\sigma_k\delta_{ij}\\
&\ \ -Q\sigma_k^{ls}\delta_{ls}b_{ij}+Q(\sigma_k^{il}
b_{jl}-\sigma_k^{js}b_{is})\\
&\ \ +\sigma_k\nabla_{ij}Q+\nabla_i\sigma_k\nabla_jQ
+\nabla_j\sigma_k\nabla_iQ-\nabla_{ij}h.
\end{align*}
Thus
\begin{align*}
\partial_tb_{ij}-Q\sigma_k^{ls}\nabla_{ls}b_{ij}
&=(k+1)Q\sigma_k\delta_{ij}-Q\sigma_k^{ls}\delta_{ls}b_{ij}
+Q(\sigma_k^{il}b_{jl}-\sigma_k^{jl}b_{il})\\
&\ \ \ \ +Q\sigma_k^{ls,mn}\nabla_jb_{ls}\nabla_ib_{mn}
+\sigma_k\nabla_{ij}Q+\nabla_j\sigma_k\nabla_{i}Q\\
&\ \ \ \ +\nabla_i\sigma_k\nabla_{j}Q-b_{ij}
\end{align*}
The evolution equation of $b^{ij}$ then follows from $\partial_tb^{ij}=-b^{im}b^{tj}\partial_tb_{mt}$, i.e.
\begin{align*}
\partial_tb^{ij}-Q\sigma_k^{ls}\nabla_{ls}b^{ij}
&=-(k+1)Q\sigma_kb^{il}b^{jl}
+Q\sigma_k^{ls}\delta_{ls}b^{ij}\\
&\ \ \ -Qb^{il}b^{js}(\sigma_k^{yl}b_{ys}-\sigma_k^{ys}b_{yl})\\
&\ \ \ -Qb^{ir}b^{ju}(\sigma_k^{ls,mn}+2\sigma_k^{lm}b^{ns})
\nabla_rb_{ls}\nabla_ub_{mn}\\
&\ \ \ -b^{il}b^{js}(\sigma_k\nabla_{ls}Q+\nabla_s
\sigma_k\nabla_lQ+\nabla_l\sigma_k\nabla_sQ)+b^{ij}.
\end{align*}

At $x_t$, we have
\begin{align}\label{3.10}
\nabla_i(\frac{b^{11}}{h})=0,\ \ i.e., \ \ b^{11}b_{11i}=-\frac{h_i}{h},
\end{align}
and
$$\nabla_{ij}(\frac{b^{11}}{h})\leq0.$$

Next, we compute the evolution equation of $\frac{b^{11}}{h}$ as
\begin{align*}
&\partial_t\frac{b^{11}}{h}-Q\sigma_k^{ij}
\nabla_{ij}\frac{b^{11}}{h}\\
&=\frac{2}{h}Q\sigma_k^{ij}
\nabla_i\frac{b^{11}}{h}\nabla_jh
+\frac{Q}{h^2}b^{11}\sigma_k^{ij}\nabla_{ij}h
-(k+1)\frac{Q}{h}\sigma_k(b^{11})^2
+\frac{Q}{h}\sigma_k^{ij}\delta_{ij}b^{11}\\
&\ \ \ \ -\frac{Q}{h}(b^{11})^2(\sigma_k^{ij,ml}
+2\sigma_k^{im}b^{lj})\nabla_1b_{ij}\nabla_1b_{ml}
-\frac{b^{11}}{h^2}Q\sigma_k+2\frac{b^{11}}{h}\\
&\ \ \ \ -\frac{1}{h}(b^{11})^2(\nabla_{11}Q\sigma_k
+2\nabla_1\sigma_k\nabla_1Q)\\
&=\frac{2}{h}Q\sigma_k^{ij}
\nabla_i\frac{b^{11}}{h}\nabla_jh
-(k+1)\frac{Q}{h}\sigma_k(b^{11})^2
-\frac{Q}{h}(b^{11})^2(\sigma_k^{ij,ml}
+2\sigma_k^{im}b^{lj})\nabla_1b_{ij}\nabla_1b_{ml}\\
&\ \ \ \ -\frac{1}{h}(b^{11})^2(\nabla_{11}Q\sigma_k
+2\nabla_1\sigma_k\nabla_1Q)+(k-1)\frac{b^{11}}{h^2}
Q\sigma_k+\frac{2b^{11}}{h}.
\end{align*}
According to inverse concavity of $(\sigma_k)^{\frac{1}{k}}$, we obtain, by results in \cite{AMZ}
\begin{align}\label{3.11}
(\sigma_k^{ij,ml}+2\sigma_k^{im}b^{lj})\nabla_1b_{ij}
\nabla_1b_{ml}\geq\frac{k+1}{k}
\frac{(\nabla_1\sigma_k)^2}{\sigma_k}.
\end{align}
On the other hand, by the Schwartz inequality, the following inequality holds
\begin{align}\label{3.12}
2|\nabla_1\sigma_k\nabla_1Q|\leq\frac{k+1}{k}
\frac{Q(\nabla_1\sigma_k)^2}{\sigma_k}
+\frac{k}{k+1}\frac{\sigma_k(\nabla_1Q)^2}{Q}.
\end{align}
Using (\ref{3.11}) and (\ref{3.12}), at $x_t$, there is
\begin{align}\label{3.13}
\partial_t\frac{b^{11}}{h}\leq-\frac{(b^{11})^2}{h}
\sigma_k[\nabla_{11}Q-\frac{k}{k+1}\frac{(\nabla_1Q)^2}{Q}
+(k+1)Q+(1-k)\frac{Qb_{11}}{h}]+\frac{2b^{11}}{h}.
\end{align}

Let $\iota$ be the arc-length of the great circle passing through $x_t$ with the unit tangent vector $e_1$, then
\begin{align}\label{3.14}
\nonumber\nabla_{11}Q-\frac{k}{k+1}\frac{(\nabla_1Q)^2}{Q}
+(k+1)Q&=(k+1)Q^{\frac{k}{k+1}}
(Q^{\frac{1}{k+1}}+(Q^{\frac{k}{k+1}})_{\iota\iota})\\
&=(k+1)Q(1+Q^{-\frac{1}{k+1}}
(Q^{\frac{k}{k+1}})_{\iota\iota}).
\end{align}
Notice that
$$\nabla_\iota Q=(f^{-1})_\iota h^{2-p}\rho^{q-n}
+(h^{2-p})_\iota\rho^{q-n}f^{-1}
+(\rho^{q-n})_\iota h^{2-p}f^{-1}$$
and
\begin{align*}
\nabla_{\iota\iota}Q&
=(f^{-1})_{\iota\iota}h^{2-p}\rho^{q-n}
+2(f^{-1})_{\iota}(h^{2-p})_\iota\rho^{q-n}
+(h^{2-p})_{\iota\iota}\rho^{q-n}f^{-1}\\
&\ \ +2(h^{2-p})_\iota(\rho^{q-n})_\iota f^{-1}
+(\rho^{q-n})_{\iota\iota}h^{2-p}f^{-1}
+2(\rho^{q-n})_\iota(f^{-1})_\iota h^{2-p}.
\end{align*}
We have by direct computations
\begin{align*}
1+Q^{-\frac{1}{k+1}}(Q^{\frac{k}{k+1}})_{\iota\iota}
&=1+\frac{Q^{-1}}{k+1}Q_{\iota\iota}-
\frac{kQ^{-2}}{(k+1)^2}Q^2_{\iota}\\
&=1+\frac{h^{p-2}}{k+1}(h^{2-p})_{\iota\iota}
+\frac{2f(f^{-1})_\iota}{k+1}[(h^{2-p})_\iota h^{p-2}
+(\rho^{q-n})_\iota \rho^{n-q}]\\
&\ \ \ \ +\frac{f}{k+1}(f^{-1})_{\iota\iota}
+\frac{2h^{p-2}\rho^{n-q}}{k+1}(h^{2-p})_\iota
(\rho^{q-n})_\iota+\frac{\rho^{n-q}}{k+1}
(\rho^{q-n})_{\iota\iota}\\
&\ \ \ \ -\frac{2kf^2}{(k+1)^2}(f^{-1})^2_{\iota}
-\frac{kfh^{p-2}\rho^{n-q}}{(k+1)^2}
[h^{p-2}(h^{2-p})_\iota^2+\rho^{n-q}(\rho^{q-n})_\iota^2]\\
&\geq\frac{h^{p-2}}{k+1}(h^{2-p})_{\iota\iota}
+\frac{2f(f^{-1})_\iota}{k+1}[(h^{2-p})_\iota h^{p-2}
+(\rho^{q-n})_\iota \rho^{n-q}]\\
&\ \ \ \ -\frac{kfh^{p-2}\rho^{n-q}}{(k+1)^2}
[h^{p-2}(h^{2-p})_\iota^2+\rho^{n-q}(\rho^{q-n})_\iota^2]\\
&\ \ \ \ +\frac{f}{k+1}(f^{-1})_{\iota\iota}
-\frac{2kf^2}{(k+1)^2}(f^{-1})^2_{\iota}\\
&\geq c_0,
\end{align*}
where $c_0$ is a positive constant depending on $f$, $h$,  $\rho$ and $k$.

Therefore, we can derive
$$\partial_t\frac{b^{11}}{h}\leq-
\bigg(\frac{b^{11}}{h}\bigg)^2Q
\sigma_k[(k+1)c_0+(1-k)\frac{b_{11}}{h}]
+\frac{2b^{11}}{h}.$$
By the uniform bounds on $f$, $h$ and $\sigma_k$, we conclude
$$\partial_t\frac{b^{11}}{h}
\leq-c_1\bigg(\frac{b^{11}}{h}\bigg)^2+\frac{2b^{11}}{h},$$
Here $c_1$ is a positive constant independent of $t$.
The maximum principle then gives the upper bound
of $b^{11}$, and the result follows.
\end{proof}

As a consequence of Lemmas \ref{la3.4}, \ref{la3.5} and \ref{la3.6}, we can obtain the following corollary.

\begin{corollary}\label{co3.7}
Under the assumptions of Lemma \ref{la3.1}, then
\begin{align*}
C_6\leq\kappa_i(\cdot,t)\leq C_5,\ \ i=1,\cdots,n-1,
\end{align*}
where $C_6$ is positive constant independent of $t$.
\end{corollary}

\section{ The long-time existence and convergence of flow (\ref{1.1})}\label{S4}
In this section we show the long-time existence and convergence of solutions to the flow (\ref{1.1}), that is, give the proof of Theorem \ref{th1.1}.
First, we prove that the functional $\mathcal{J}(M_t)$ defined below is monotone non-decreasing along the flow (\ref{1.1}).

\begin{lemma}\label{la4.1}
Define the functional
$$\mathcal{J}(M_t)=\frac{1}{q+k-n}\int_{\mathbb{S}^{n-1}}
\rho^{q+k-n}(u,t)du
-\frac{1}{p+1}\int_{\mathbb{S}^{n-1}}f(x)h^{p+1}(x,t)dx.$$
Let $p,q\in\mathbb{R}$ and $k=0,\cdots,n-1$.
Then $\mathcal{J}(M_t)$ is monotone non-decreasing
along the flow (\ref{1.1}).
\end{lemma}

\begin{proof}
From (\ref{2.3}) and (\ref{2.3.1}), by the fact that $\rho(u,t)du=h(x,t)dx$, we have
\begin{align*}
\partial_t\mathcal{J}&=\int_{\mathbb{S}^{n-1}}
\rho^{q+k-n-1}(u,t)\partial_t\rho(u,t)du -\int_{\mathbb{S}^{n-1}}f(x)h^{p}(x,t)\partial_th(x,t)dx\\
&=\int_{\mathbb{S}^{n-1}}
\rho^{q+k-n-1}\frac{\rho}{h}\partial_thdu
-\int_{\mathbb{S}^{n-1}}f(x)h^{p}\partial_thdx\\
&=\int_{\mathbb{S}^{n-1}}
\rho^{q-n}h\sigma_k\partial_thdx
-\int_{\mathbb{S}^{n-1}}f(x)\frac{h}{h^{1-p}}\partial_thdx\\
&=\int_{\mathbb{S}^{n-1}}
\bigg(\frac{f^{-1}\rho^{q-n}h^{2-p}
\sigma_k-h}{h^{1-p}}\bigg)\partial_thdx\\
&=\int_{\mathbb{S}^{n-1}}\frac{(f^{-1}\rho^{q-n}h^{2-p}
\sigma_k-h)^2}{h^{1-p}}dx\\
&\geq0.
\end{align*}
Clearly $\partial_t\mathcal{J}=0$ holds if and only if
$h^{1-p}\rho^{n-q}\sigma_k=f.$
\end{proof}

\subsection*{Proof of Theorem \ref{th1.1}}
From the Corollary \ref{co3.7}, we see that the equation
(\ref{2.3}) is uniformly parabolic and has the short time existence.
By $C^0, C^1$ and $C^2$-estimates, and the Krylov's theory \cite{K}, we get the H\"{o}lder continuity of $\nabla^2h$ and $\partial_th$.
Then we get the higher order derivation estimates by the regularity theory of the uniformly parabolic equations. Therefore, we obtain the long-time existence and regularity of the solution to the flow (\ref{2.3}).
Moreover, we have
$$\|h\|_{C_{\xi,t}^{i,j}(\mathbb{S}^{n-1}
\times[0,T))}\leq C$$
for some $C>0$, independent of $t$, and for each pairs of nonnegative integers $i$ and $j$.

With the aid of the Arzel\`{a}-Ascoli theorem and a diagonal argument, there exists a sequence of $t$, denoted by $\{t_k\}_{k\in\mathbb{N}}\subset(0,\infty)$, and a smooth function $h(x)$ such that
$$\|h(x,t_k)-h(x)\|_{C^i(\mathbb{S}^{n-1})}\rightarrow0$$
uniformly for any nonnegative integer $i$ as  $t_k\rightarrow\infty$.
This illustrates that $h(x)$ is a support function. Let $M$ be a convex body determined by $h(x)$, we conclude that $M$ is smooth and strictly convex with the origin in its interior.

In the following, we prove that Equation (\ref{1.0}) has a smooth solution.
From Lemma \ref{la4.1}, the functional
$\mathcal{J}(M_t)$
is non-decreasing along the flow. Then by Lemma \ref{la3.1}, we have
$$\mathcal{J}(M_t)\leq C,\ \  t\in[0,T).$$
The above facts show that
$$\int_0^t\mathcal{J}^\prime(M_t)dt=\mathcal{J}(M_t)
-\mathcal{J}(M_0)\leq \mathcal{J}(M_t)\leq C,$$
which leads to
$$\int_0^\infty \mathcal{J}^\prime(M_t)dt\leq C.$$
This means that there exists a subsequence of times $t_j\rightarrow\infty$ such that
$$\mathcal{J}^\prime(M_{t_j})\rightarrow0.$$

From the proof of Lemma \ref{la4.1}, by passing to the limit, we obtain
\begin{align*}
\int_{\mathbb{S}^{n-1}}\frac{(f^{-1}
\tilde{\rho}^{q-n}\tilde{h}^{2-p}
\tilde{\sigma}_k-\tilde{h})^2}{\tilde{h}^{1-p}}dx=0
\end{align*}
where $\tilde{\rho}$, $\tilde{h}$ and $\tilde{\sigma}_k$ are the support function, radial function and $k$-th elementary symmetric function for principal curvature radii of $M_\infty$. Therefore
$$\frac{\tilde{h}^{1-p}}{(\tilde{h}^2+|\nabla \tilde{h}|^2)^{\frac{n-q}{2}}}\tilde{\sigma}_k(x)=f(x).$$
The proof of Theorem \ref{th1.1} is completed.                                \hfill$\square$

\section{ The uniqueness of solution to the equation (\ref{1.2})}\label{S5}

In this section, the uniqueness of the solution to Equation (\ref{1.0}) can be obtained, namely Theorem \ref{th1.4}.

\subsection*{Proof of Theorem \ref{th1.4}}
Let $h_1$ and $h_2$ be two solutions of
\begin{align}\label{5.1}
\frac{h^{1-p}}{(h^2+|\nabla h|^2)^{\frac{n-q}{2}}}\sigma_k(\nabla^2h+hI)=f(x).
\end{align}
We first need to prove
\begin{align}\label{5.2}
\max\frac{h_1}{h_2}\leq1
\end{align}
by contradiction.

Suppose (\ref{5.2}) is not true, namely $\max\frac{h_1}{h_2}>1$. Assume $\frac{h_1}{h_2}$ attains its maximum at $z_0\in\mathbb{S}^{n-1}$, then $h_1(z_0)>h_2(z_0)$. Let $\mathcal{L}=\log\frac{h_1}{h_2}$, then at $z_0$
$$0=\nabla \mathcal{L}=\frac{\nabla h_1}{h_1}-\frac{\nabla h_2}{h_2},$$
and
$$0\geq\nabla^2 \mathcal{L}=\frac{\nabla^2 h_1}{h_1}-\frac{\nabla^2 h_2}{h_2}.$$
By Equation (\ref{5.1}) and the $k$-degree homogeneity of $\sigma_k$, we have at $z_0$
\begin{align*}
1&=\frac{\frac{h_2^{1-p}}{(h_2^2+|\nabla h_2|^2)^{\frac{n-q}{2}}}\sigma_k(\nabla^2h_2+h_2I)}
{\frac{h_1^{1-p}}{(h_1^2+|\nabla h_1|^2)^{\frac{n-q}{2}}}\sigma_k(\nabla^2h_1+h_1I)}\\
&=\frac{h_2^{1-p-n+q+k}\sigma_k(\frac{\nabla^2h_2}{h_2}+I)}
{h_1^{1-p-n+q+k}\sigma_k(\frac{\nabla^2h_1}{h_1}+I)}\\
&\geq\frac{h_2^{1-p-n+q+k}\sigma_k(\frac{\nabla^2h_2}{h_2}+I)}
{h_1^{1-p-n+q+k}\sigma_k(\frac{\nabla^2h_2}{h_2}+I)}\\
&=\frac{h_2^{1-p-n+q+k}}{h_1^{1-p-n+q+k}}.
\end{align*}
Since $q<p$, it follows that $h_2(z_0)\geq h_1(z_0)$.
This is a contradiction. Thus (\ref{5.2}) holds.

Interchanging $h_1$ and $h_2$, (\ref{5.2}) implies
$$\max\frac{h_2}{h_1}\leq1.$$
Combining it with (\ref{5.2}), we have $h_1\equiv h_2$.
 \hfill$\square$

\subsection*{Acknowledgments}

The authors would like to express their heartfelt thanks to Professor H. Li, Professor X. Zhang, Q. Ding, Professor Y. Liu and Professor J. Lu for helpful comments and suggestions.

\end{document}